\input ./style/arxiv-vmsta.cfg
\documentclass[numbers,compress,v1.0.1]{vmsta}
\usepackage{bm}
\usepackage{vtexurl}
\usepackage{mathrsfs}

\volume{4}
\issue{3}
\pubyear{2017}
\firstpage{253}
\lastpage{271}
\doi{10.15559/17-VMSTA86}

\startlocaldefs

\allowdisplaybreaks

\newtheorem{theorem}{Theorem}

\newtheorem{proposition}{Proposition}
\newtheorem{corollary}{Corollary}
\newtheorem{lemma}{Lemma}
\newtheorem{claim}{Claim}[proposition]

\theoremstyle{definition}

\newtheorem{remark}{Remark}

\def\bE{\mathbb{E}}
\def\bN{\mathbb{N}}
\def\bS{\mathbb{S}}
\def\bR{\mathbb{R}}
\def\bZ{\mathbb{Z}}
\def\bp{\bm{p}}

\DeclareMathOperator{\supp}{supp}

\endlocaldefs

\begin{document}
\begin{frontmatter}

\title{Random iterations of homeomorphisms\\ on the circle}

\author[a]{\inits{K.}\fnm{Katrin}\snm{Gelfert}\fnref{f1}}\email{gelfert@im.ufrj.br}
\author[b]{\inits{\"O.}\fnm{\"Orjan}\snm{Stenflo}\corref{cor1}\fnref{f1}}\email{stenflo@math.uu.se}
\cortext[cor1]{Corresponding author.}
\address[a]{Institute of Mathematics, Federal University of Rio de
Janeiro, 22.453 Rio de Janeiro RJ, Brazil}
\address[b]{Department of Mathematics, Uppsala University, Box 480,
75106 Uppsala, Sweden}

\fntext[f1]{KG has been supported, in part, by CNPq research grant
302880/2015-1 (Brazil). KG and \"OS thank ICERM (USA) for their
hospitality and financial support.}

\dedicated{Dedicated to Professor Dmitrii S.\ Silvestrov on the occasion
of his 70th Birthday}

\markboth{K. Gelfert, \"O. Stenflo}{Random iterations of homeomorphisms
on the circle}

\begin{abstract}
We study random independent and identically distributed iterations of
functions from an iterated function system of homeomorphisms on the
circle which is minimal. We show how such systems can be analyzed in
terms of iterated function systems with probabilities which are
non-expansive on average.
\end{abstract}

\received{22 March 2017}
\revised{25 September 2017}
\accepted{25 September 2017}
\publishedonline{5 October 2017}
\end{frontmatter}


\section{Introduction}

We study iterations of a finite family of circle homeomorphisms.
This topic has been studied already from a number of different points
of view.
One may, for example, take a purely deterministic approach and study
the associated action of the \emph{group of circle homeomorphisms} (the
special case of the group of orientation preserving circle
diffeomorphisms is treated
in~\cite{Ghy:01,Nav:11,GGKV14}).
Or one may, as we will, take a probabilistic approach and investigate
Markov chains generated by random independent and identically
distributed (i.i.d.) iterations of functions from the family (such as
in~\cite{LeJan87,Cra:90,SZ16}).

We restrict our attention to families of functions which are forward minimal
in the sense that for any two points on the circle, there are orbits
from the first point arbitrary close to the second one
using some concatenations of functions from the family.
The set of distances which are preserved simultaneously
by all maps allows us to distinguish between distinct types of ergodic
behavior for such Markov chains.

By finding a topologically conjugate system which is non-expansive on average,
under the additional assumption that the system of inverse maps is
forward minimal, we prove limit theorems
including almost sure synchronization of random trajectories (which is
sometimes also referred to as Antonov's theorem~\cite{Ant:84}) provided
that the system is not topologically conjugate to a family containing
only isometries, and uniqueness and fiberwise properties of stationary
distributions.

In contrast to many previous authors we do not assume that all maps
preserve orientation or, \emph{a priori},
that the system of inverse maps is forward minimal (such as in~\cite
{Ant:84,Ghy:01,KN04,Nav:11,GGKV14,SZ16}) or contains at least one map
which is minimal (as in~\cite{SZ16}). Our setting is also studied
in~\cite{M17} (without any minimality condition), where a different
approach is
used and ideas of~\cite{AviVia:10} are adapted which in turn are built
on ideas of~\cite{Led:86,Cra:90}. See also \cite{SZ17}. One further
precursor in a more specific setting is the work by Furstenberg~\cite
{Fur:63} where the homeomorphisms are the projective actions of
elements of $SL_2(\bR)$.

\section{Random iterations}

Let $K$ be a compact topological space equipped with its Borel sets. We
call a finite set $F=\{f_1,\ldots,f_N\}$ of continuous
functions $f_j\colon K\to K$, $j=1,\ldots,N$, an \emph{iterated
function system (IFS)}.
If all maps $f_j$ are homeomorphisms, as we will in general assume
here, then we also consider the associate
IFS $F^{-1}:=\{f_1^{-1},\ldots,f_N^{-1}\}$ of the inverse maps.

We will discuss different points of view on random and deterministic
iterations of functions from an IFS and recall some standard notations
and facts.

Given $(I_n )_{n \geq1}$ a stochastic sequence with values in $\{
1,\ldots,N\}$, for $x \in K$ define
\[
Z_n^x := (f_{I_n} \circ\cdots\circ
f_{I_1}) (x), \qquad Z_0^x=x.
\]
We may consider without loss of generality 
the (\emph{a priori}) unspecified common domain of the random variables
$I_n$ as $\varSigma=\{1,\ldots,N\}^\mathbb{N}$, equipped with a
probability measure $P$ defined
on its Borel subsets,
with $I_n$ being defined as $I_n(\omega)=\omega_n$
for
every $\omega= (\omega_1\omega_2\ldots)\in\varSigma$ and $n \geq1$.

We will later also consider the shift map $\sigma\colon\varSigma\to\varSigma$
defined by $\sigma(\omega_1\omega_2\ldots):=(\omega_2\omega_3\ldots)$.

For any $\omega=(\omega_1 \omega_2 \ldots) \in\varSigma$, any $n \geq0$
and any $x\in K$ we thus define $Z_n^x(\omega)=Z_n(x,\omega)$, where
\begin{equation}
\label{def:Zn} Z_n(x,\omega) := (f_{\omega_n} \circ\cdots\circ
f_{\omega_1}) (x), \qquad Z_0(x,\omega)=x.
\end{equation}

The sequence $(Z_n(x,\omega))_{n \geq0}$ is called the \emph
{trajectory} corresponding to the \emph{realization} $\omega$ of the
random process $(Z_n^x)_{n \geq0}$ starting at $x \in K$.
It is common to also consider iterates in the reversed order and to define
\begin{equation}
\label{def:hatZn} \widehat{Z}_n(x,\omega) :=(f_{\omega_1} \circ\cdots
\circ f_{\omega_n}) (x), \qquad\widehat{Z}_0(x,\omega)=x.
\end{equation}
If $F$ is an IFS of homeomorphisms, then we also consider the associate sequence
$(Z_n^-(x,\omega))_{n\ge0}$ defined by
\[
Z_n^-(x,\omega) :=\bigl(f_{\omega_n}^{-1} \circ\cdots
\circ f_{\omega_1}^{-1}\bigr) (x),\qquad Z^-_0(x,\omega)=
x,
\]
and the sequence $(\widehat Z_n^-(x,\omega))_{n\ge1}$ defined by
\begin{equation}
\label{def:hatZn-} \widehat{Z}^{-}_n(x,\omega) :=
\bigl(f_{\omega_1}^{-1} \circ\cdots\circ f_{\omega_n}^{-1}
\bigr) (x),\qquad \widehat{Z}^{-}_0(x,\omega)= x.
\end{equation}
Note that for every $\omega\in\varSigma$ and $x\in K$ it holds
\[
\widehat{Z}^{-}_n(x,\omega) 
=(f_{\omega_n}
\circ\cdots\circ f_{\omega_1})^{-1}(x) 
\quad\text{and}
\quad 
\widehat{Z}_n(x,\omega) 
=
\bigl(f_{\omega_n}^{-1} \circ\cdots\circ f_{\omega_1}^{-1}
\bigr)^{-1}(x).
\]


\subsection{Iterated function systems with probabilities and Markov chains}

Let $(I_n )_{n \geq1}$ be i.i.d. variables.
The probability measure $P$ is then a Bernoulli measure determined by
a probability vector $\bp=(p_1,\ldots,p_N)$. It then follows that
$Z_n^x=Z_n(x,\cdot)$
defined in~\eqref{def:Zn} and $\widehat{Z}_n^x=\widehat{Z}_n(x,\cdot)$
defined in~\eqref{def:hatZn} both have the same distribution for any
fixed $n \geq1$,
and $(Z_n^x)_{n \geq0}$ is a (time-homogeneous) Markov chain with
transfer operator
$T$ defined for bounded measurable functions $h\colon K \to\mathbb{R}$ by
\begin{equation}
\label{def:T} T h(x) := \sum_{j=1}^N
p_j h\bigl(f_j(x)\bigr).
\end{equation}

If $\bp$ is \emph{non-degenerate}, that is, if 
$p_j>0$ for every $j=1,\ldots,N$, then we call the pair $(F,\bp)$ an
\emph{IFS with probabilities}. The Markov chain $(Z_n^x)_{n \geq0}$ is
obtained by independent random iterations where in each iteration step
the functions $f_j$ are chosen with probability $p_j$.

Markov chains generated by IFSs with probabilities is a particular
class of Markov chains that has
received a considerable attention in recent years. The IFS terminology
was coined by Barnsley and Demko~\cite{BD85}.\footnote{A common abuse
of notation is to use the term ``IFS'' for the Markov chain $(Z_n^x)_{n
\geq0}$ obtained from an IFS with probabilities. We here stress the
deterministic nature of an IFS and the fact that an IFS can be used to
build other objects like e.g.\
$(\widehat{Z}_n(x,\omega))_{n \geq0}$. A common way to construct
fractal sets is for example to regard them
as sets of limit points for the latter sequence (assuming conditions
such as, for example, contractivity ensuring the limit to exist).}

A Borel probability measure $\mu$ on $K$ is an \emph{invariant
probability measure} for the IFS with probabilities $(F,\bp)$ if
\[
T_\ast\mu=\mu, \quad\text{where}\quad T_{\ast} \mu(\cdot) =
\sum_{j} p_j \mu\bigl(f_j^{-1}(
\cdot)\bigr).
\]
Such a measure $\mu$ is also called a \emph{stationary distribution}
for the corresponding Markov chain, since if $X$ is a $\mu$-distributed
random variable, independent of $(I_n)_{n \geq0}$ then $(Z_n^X)_{n
\geq0}$ will be a stationary stochastic sequence.

\begin{remark}\label{rem:existamea}
By continuity of all functions $f_j$, $j=1,\ldots,N$, it follows that
$(Z_n^x)_{n \geq0}$ has the \emph{weak Feller property}, that is, $T$
maps the space of real valued continuous functions on $K$ to itself.
It is well known that Markov chains with the weak Feller property have
at least one stationary distribution, see for example \cite{MT93}.
Hence, any IFS with probabilities $(F,\bp)$ has at least one invariant
probability measure.
\end{remark}

\begin{remark}
Another formalism (which will not be used here) for analyzing
stochastic sequences related to an IFS with probabilities is the one of
a (deterministic) step skew product map $(\omega,x)\mapsto(\sigma(\omega
),f_{\omega_1}(x))$ with the shift map $\sigma\colon\varSigma\to\varSigma$ in
the base and locally constant fiber maps. The Bernoulli measure is a
$\sigma$-invariant measure in the base. Invariant measures (and hence
stationary distributions) are closely related to measures which are
invariant for the step skew product (see, for example,~\cite[Chapter
5]{Via:14}).
\end{remark}

Given a positive integer $n$, define by $T^n=T\circ\cdots\circ T$ and
$T_\ast^n=T_\ast\circ\cdots\circ T_\ast$ (each $n$ times) the
concatenations of $T$ and $T_\ast$, respectively.
We call a stationary distribution $\mu$ for $(Z_n^x)_{{n\ge0}}$ \emph
{attractive} if for any $x \in K$ we have $T_\ast^n \delta_x \to\mu$
as $n \to\infty$ in the weak$\ast$ topology, where $\delta_x$ denotes
the Dirac measure concentrated in $x \in K$. In other words, for any
continuous $h \colon K\to\bR$ and for any $x \in K$ we have
\begin{eqnarray}
\label{aconv} \lim_{n\to\infty} T^n h(x) = \int h \,d
\mu.
\end{eqnarray}
An attractive stationary distribution is uniquely stationary.

Let $\rho$ be some metric on $K$.
We say that an IFS with probabilities $(F,\bp)$ is \emph{contractive on
average} with respect to $\rho$ if for any $x,y \in\mathbb{S}^1$ we have
\begin{equation}
\label{0214} \sum_{j=1}^{N}
p_j \rho\bigl(f_j(x),f_j(y)\bigr) \leq c
\rho(x,y),
\end{equation}
for some constant $c <1$ and
\emph{non-expansive on average} if (\ref{0214}) holds for some
constant $c \leq1$.

\begin{remark}\label{rem:unistamea}
It is well known that a Markov chain $(Z_n^x)_{{n\ge0}}$ generated by
an IFS with probabilities $(F,\bp)$ which is contractive on average
has an attractive (and hence unique) stationary distribution.
More generally the distribution of $Z_n^x$ then converges (in the
weak$\ast$ topology) to the stationary distribution with an exponential
rate that can be quantified for example by the Wasserstein metric,
see e.g. \cite{S12}.

Far less is known for non-expansive systems. The theory for Markov
chains generated by non-expansive systems can be regarded as belonging
to the realm of Markov chains where $\{T^n h\}$ is equicontinuous for
any continuous $h \colon K \to \mathbb{R}$, or ``stochastically
stable'' Markov chains (see \cite{MT93} for a survey).
\end{remark}

The Markov chain $(Z_n^x)_{n \geq0}$ is \emph{topologically recurrent}
if for any open set $O\subset K$ and any $x \in K$ we have
\[
P \bigl(Z_n^x \in O \text{ for some } n \bigr) > 0.
\]

In the present paper we are going to study a special class of
topologically recurrent Feller continuous Markov chains generated by
IFSs with probabilities of homeomorphisms on the circle. The topology
of the circle and the hence implied monotonicity of the maps play a
crucial role for our results.

\section{IFSs with homeomorphisms on the circle}

From now on we will always assume $K=\bS^1= \bR\slash\bZ$ to be the
unit circle and consider an IFS $F=\{f_j\}_{j=1}^N$ of homeomorphisms
$f_j\colon\bS^1\to\bS^1$. Let $d(x,y):=\min\{ |y-x|, 1-|y-x| \}$ be the
standard metric on $\bS^1$.

\subsection{Deterministic iterations and simultaneously preserved distances}

An IFS $F=\{f_j\}_{j=1}^N$ is \emph{forward minimal} if for any open
set $O \subset K$ and any $x \in K$ there exist some $n \geq0$ and
some $\omega\in\varSigma$ such that
\[
Z_n(x,\omega) \in O.
\]
In other words, for a forward minimal IFS it is possible to go from any
point $x$ arbitrarily close to any point $y$ by applying some
concatenations of functions in the IFS.
We say that the IFS $F=\{f_j\}_{j=1}^N$ of homeomorphisms $f_j$ is \emph
{backward minimal} if the IFS $\{f_j^{-1}\}_{j=1}^N$ is forward minimal.

\begin{remark}
Note that $F$ is forward (backward) minimal if and only if for every
nonempty closed set $A\subset\bS^1$ satisfying $f_j(A)\subset A$
($f_j^{-1}(A)\subset A$) for every $j$ we have $A=\bS^1$.

Note that not every forward minimal IFS is automatically backward
minimal if $N>1$ (see~\cite{BGMS17} for a discussion and counterexamples).
By \cite[Corollary E]{BGMS17}, an IFS is both forward and backward
minimal if and only if there exists an $\omega\in\varOmega$ such that
$(Z_n(x,\omega))_{n \geq0}$ is dense, for any $x \in\mathbb{S}^1$.
(By forward minimality this property trivially holds for some fixed $x
\in\mathbb{S}^1$, but the choice of $\omega$ might depend on $x \in
\mathbb{S}^1$.)
A simple sufficient condition for an IFS of circle homeomorphisms to be
both forward and backward minimal is that at least one of the maps has
a dense orbit. A class of IFSs which are forward and backward minimal
(so-called expanding-contracting blenders) but without a map with a
dense orbit can be found in \cite[Section~8.1]{DGR17}.
\end{remark}

The following is somehow related to the study of the well-known concept
of \emph{rotation numbers} of orientation-preserving circle
homeomorphisms which was introduced by Poincar\'e and which provides an
invariant to (almost completely) characterize topologically conjugacy.%
\footnote{The rotation number $r(f)$ of a circle homeomorfism $f$ is
rational if, and only if, $f$ has a periodic orbit. If $r(f)$ is
irrational then $f$ is semi-conjugate to a rotation by angle $r(f)$
and, in particular, this semi-conjugacy is a conjugacy if $f$ is minimal.}
Rotation numbers are also important when studying an IFS (which can be
considered as a special group action) of orientation-preserving circle
homeomorphisms. The surveys~\cite{Ghy:01,Nav:11} review these facts,
see also~\cite{GGKV14}.

Here we deal with a more general class of IFSs in which not necessarily
all maps preserve orientation.

Given $F$ and a metric $\rho$ on $\bS^1$, let $L=L(F,\rho)$ defined by
\begin{equation}
\label{def:L} %
\begin{split} L := \bigl\{ s \in[0,1/2]\colon&
\rho(x,y)=s \ \text{implies that}\ \rho\bigl(f_j(x),f_j(y)
\bigr)=s
\\
& \text{ for any } j=1,\ldots, N \text{ and } (x,y) \in\bS^1\times
\bS^1\bigr\} \end{split} %
\end{equation}
be the set of $\rho$-distances which simultaneously are preserved by
all maps in $F$.

\begin{remark}
Note that since all maps of the IFS are homeomorphisms it follows that
for every $x,y \in\mathbb{S}^1$ with $\rho(x,y) \in L(F,\rho)$ we have
\[
\rho(x,y)=\rho\bigl(f_j(x),f_j(y)\bigr)= \rho
\bigl(f_j^{-1}(x),f_j^{-1}(y)\bigr)
\quad\text{for all}\ j=1,\ldots,N,
\]
and thus $L(F,\rho)=L(F^{-1},\rho)$. Moreover, note that by continuity
of the maps of the IFS, the set $L$ is closed.
\end{remark}

We have the following dichotomy.

\begin{lemma}\label{lem:L}
If $L=L(F,\rho)$ is finite, then
\[
L= \biggl\{0,\frac{1}{k},\frac{2}{k},\ldots,\frac{\lfloor{k/2} \rfloor
}{k} \biggr
\},
\]
for some $k \geq1$.

If $L=L(F,\rho)$ is infinite, then $L=[0,1/2]$. All IFS maps are then
isometries (with respect to $\rho$).
\end{lemma}

\begin{proof}
Consider the operation $\oplus\colon L\times L\to\mathbb{S}^1$ defined by
\[
s_1 \oplus s_2 := \min\{ s_1+s_2,1-s_1-s_2
\}.
\]
Note that $L$ is closed under this operation, that is, $\oplus:(L\times
L) \to L$.
Indeed, given $s_1,s_2\in L$, if $x,z\in\mathbb{S}^1$ are such that
$\rho(x,z)=s_1\oplus s_2$, then there is a point $y \in\mathbb{S}^1$
such that
$\rho(x,y)=s_1 $ and $\rho(y,z)=s_2$.
Thus, we have $\rho(f_j(x),f_j(y))=s_1$ and $\rho(f_j(y),f_j(z))=s_2$
for every $j=1,\ldots,N$. Since all maps $f_j$ are homeomorphisms, it
follows that $\rho(x,z)=\rho(f_j(x),f_j(z))$ for all $j=1,\ldots,N$ and
hence $s_1\oplus s_2\in L$.

It follows that if $L$ is finite (and nontrivial) then the smallest
positive element of $L$ must be a rational number of the form $1/k$ for
some integer $k >1$ and hence $L$ must have the given form.

If $L$ is infinite, then $L=[0,1/2]$, since $L$ has then arbitrary
small positive elements and must therefore be a dense, and by
continuity of all maps in $F$, also a closed subset of $[0,1/2]$. All
IFS maps are then isometries.
\end{proof}

\begin{remark}\label{rem:00}
If $L(F,d)$ is finite and $1/k$ is its smallest positive element, then
the IFS $\widetilde F=\{\tilde f_j\}$ with maps $\tilde f_j(x)=k
(f_j(x/k)\ \text{mod } 1/k)$, $j=1, \ldots, N$, satisfies $L(\widetilde
F,d)=\{0\}$. Thus, we can describe the dynamical properties of an IFS
with the set of preserved distances $L(F,d)$ being finite in terms of
the dynamics of an IFS with no positive preserved distances. Observe
that each of the maps $f_j$ is semiconjugate with $\tilde f_j$ by means
of the map $\pi\colon\bS^1\to\bS^1$ defined by $\pi(x)=kx\mod1$, that
is, we have $\pi\circ f_j=\tilde f_j\circ\pi$.
\end{remark}

Intuitively we may in all cases regard the infimum of all positive
elements of
$L=L(F,d)$ as the ``common prime period'' of all maps, where the case
when $L$ is infinite corresponds to a degenerated case. As mentioned
above, for orientation-preserving homeomorphisms this number can be
compared with the rotation number functions in~\cite{Ghy:01,Nav:11,GGKV14}.

\subsection{Random iterations}

First, recall the following well-known fact about
forward
minimal IFSs with probabilities on $\mathbb{S}^1$ (compare also~\cite
[Lemma 2.3.14]{Nav:11}).
We say that a measure $\mu$ has \emph{full support} if the support of
$\mu$ is $\bS^1$.

\begin{lemma}\label{lem:exiinvmea}
Let $(F,\bp)$ be an IFS with probabilities of homeomorphisms on
$\mathbb{S}^1$ and $\mu_+$ be an invariant probability measure for
$(F,\bp)$. If $F$
is forward minimal
then $\mu_+$ is nonatomic and has full support.
\end{lemma}
\begin{proof}
By contradiction, suppose that $\mu_+$ is atomic. Let $x\in\bS^1$ be a
point of maximal positive $\mu_+$-mass. By invariance of $\mu_+$, we obtain
\[
\mu_+\bigl(\{x\}\bigr) = \sum_{j=1}^Np_j
\mu_+\bigl(\bigl\{f_j^{-1}(x)\bigr\}\bigr)
\]
and hence, since we assume that $\bp$ is non-degenerate, we have $\mu
_+(\{f_j^{-1}(x)\})=\mu_+(\{x\})$ for every $j$.
Hence, we obtain that the (nonempty) set
\[
A :=\bigl\{y\in\bS^1\colon\mu_+\bigl(\{y\}\bigr)=\mu_+\bigl(\{x\}
\bigr)\bigr\}
\]
satisfies $f_j^{-1}(A)\subset A$ for every $j$. Since $\mu_+$ is
finite, $A$ is finite (and, in particular, closed). Hence, since every
$f_j^{-1}$ is bijective, we in fact have $f_j^{-1}(A)=A$ and $f_j(A)=A$
for every $j$.
Assuming that $F$ is either backward minimal or forward minimal, we
hence obtain $A=\bS^1$, which is a contradiction. Hence $\mu_+$ is nonatomic.

An analogous argument shows that $\mu_+$ has full support. Indeed,
let the (closed) set $A=\supp\mu_+$ denote the support of $\mu_+$.
By invariance of $\mu_+$, for every $j$ we have $\mu_+(f_j^{-1}(A))=\mu
_+(A)=1$ which implies
$A \subset f_j^{-1}(A)$, i.e.\ $f_j(A) \subset A$ for
every $j$, so if $(F,\bp)$ is forward minimal, then $\mu_{+}$ has full support.
\end{proof}

We say that a probability measure $\mu$ on $\mathbb{S}^1$ is \emph
{$s$-invariant} for $s\in[0,1]$ if $(R_s)_\ast\mu= \mu$, where $R_s(x)=
(x + s) \text{ mod } 1$. Analogously, we say that an $\mathbb
{S}^1$-valued random variable $X$ is \emph{$s$-invariant} if its
distribution is $s$-invariant, in which case $X$ and $R_s(X)$ have the
same distribution.

\begin{lemma}\label{lem:1kinv-b}
Let $(F,\bp)$ be an IFS with probabilities of homeomorphisms on $\bS
^1$ which is forward minimal. Then any invariant probability measure
for $(F,\bp)$ is $s$-invariant for any $s\in L(F,d)$.
\end{lemma}

\begin{proof}
Let $\mu$ be an invariant probability measure for $(F,\bp)$. Let $s\in L(F,d)$.
Consider an arbitrary interval $I$ of length $s$ satisfying $\mu(I)
\geq\mu(I')$ for all other intervals $I'$ of length $s$. By invariance
of $\mu$ we have $\mu(I)=\sum_{j} p_j \mu(f_j^{-1}(I))$. Hence, since
$\bp$ is non-degenerate, it follows that
$\mu(I)=\mu(f_j^{-1}(I))$ for every $j$.

Since $I$ is of length $s\in L(F,d)=L(F^{-1},d)$, the interval
$f_j^{-1}(I)$ is also of length $s$ for any $j$.
More generally, the $\mu$-measure of the image of $I$ under arbitrary
finite concatenations of functions from $F^{-1}$ is an interval of
length $s$ and of measure $\mu(I)$.
By forward minimality and continuity of the maps in $F$ it therefore
follows that all intervals of length $s$ have the same $\mu$-measure
equal to $\mu(I)$.

This property implies that $\mu$ is $s$-invariant. Indeed,
consider an arbitrary interval $(c,d)$ in $\mathbb{S}^1$, where
$d=R_\alpha(c)$, for some $0<\alpha\leq1/2$. If $\alpha$ is larger
than $s$ then
\[
\begin{split} \mu((c,d)) &=\mu\bigl(\bigl(c,R_{s}(c)
\bigr)\bigr)+ \mu\bigl(\bigl(R_{s}(c),d\bigr)\bigr) = \mu\bigl(
\bigl(d,R_{s}(d)\bigr)\bigr)+ \mu\bigl(\bigl(R_{s}(c),d
\bigr)\bigr)
\\
&=\mu\bigl(\bigl(R_{s}(c),R_{s}(d)\bigr)\bigr). \end{split}
\]
Otherwise, if $\alpha$
is smaller than or equal to $s$, then
\[
\begin{split} \mu((c,d))+ \mu\bigl(\bigl(d,R_{s}(c)
\bigr)\bigr) &= \mu\bigl(\bigl(c,R_{s}(c)\bigr)\bigr) =\mu(I) = \mu\bigl(
\bigl(d,R_{s}(d)\bigr)\bigr)
\\
&= \mu\bigl(\bigl(d,R_{s}(c)\bigr)\bigr)+ \mu\bigl(
\bigl(R_{s}(c),R_{s}(d)\bigr)\bigr), \end{split} %
\]
which also implies $\mu((c,d))=\mu((R_{s}(c),R_{s}(d))$.
\end{proof}

Given a measurable transformation $\varPhi\colon\bS^1\to\bS^1$ and a
probability measure $\mu$, we denote by $\varPhi_\ast\mu$ the \emph
{pushforward} of $\mu$ defined by $\varPhi_\ast\mu(E)=\mu(\varPhi^{-1}(E))$
for each Borel set $E$ of $\bS^1$.

\begin{remark}
Recall that if $\mu$ is nonatomic (i.e.\ continuous) and fully
supported Borel measure on $\bS^1$ then its distribution function
defines a homeomorphism $\varPhi\colon\bS^1\to\bS^1$ and $\varPhi^{-1}_\ast\mu
=\mu_{\rm Leb}$.
\end{remark}

We state a preliminary result.%
\footnote{The main idea is well known (see, for example,~\cite[p.
118]{LeJan87} and~\cite{GGKV14}, where the authors also consider a
measurable bijection analogous to the here defined conjugation map $\varPhi_{-}$).}

\begin{proposition} \label{avcontr}
Let $(F,\bp)$ be an IFS with probabilities of homeomorphisms on $\bS
^1$ which is
backward
minimal.
Let $\mu_-$ be an invariant measure for $(F^{-1},\bp)$ and let
$\varPhi_{-}\colon\bS^1\to\bS^1$ be defined by $\varPhi_{-}(x):=\mu_-([0,x])$.
Then
\[
\rho(x,y):=\min\bigl\{ \mu_{-}\bigl([x,y]\bigr), \mu_{-}
\bigl([y,x]\bigr)\bigr\}
\]
is a metric on $\bS^1$ and $(F,\bp)$ is non-expansive on average with
respect to $\rho$.

The IFS $G=\{g_j\}_{j=1}^N$ given by the maps $g_j:=\varPhi_{-}\circ
f_j\circ\varPhi_{-}^{-1}$, $j=1,\ldots,N$, with probabilities $\bp$ is
non-expansive on average with respect to $d$ and we have
$L(G,d)=L(F,\rho)$.
\end{proposition}

\begin{proof}
Let $(F,\bp)$ be an IFS with probabilities of homeomorphisms on $\bS^1$
which is backward
minimal.
Let $\varPhi_{-}(x)=\mu_{-}([0,x])$, where $\mu_-$ is an invariant
probability measure for $(F^{-1},\bp)$, and define
\[
\rho(x,y):=\min\bigl\{ \mu_{-}\bigl([x,y]\bigr), \mu_{-}
\bigl([y,x]\bigr)\bigr\}.
\]
Clearly, $L(G,d)=L(F,\rho)$.
By Lemma~\ref{lem:exiinvmea} applied to $(F^{-1},\bp)$, $\mu_{-}$ is
nonatomic and has full support 
and hence we have $\rho(x,y)\ge0$ and $\rho(x,y)=0$ if and only if
$x=y$. Moreover, clearly $\rho(x,y)=\rho(y,x)$ and $\rho(x,y)\le\rho
(x,z)+\rho(z,y)$. Hence, $\rho$ defines a metric on $\bS^1$.
The definition of $\rho$ and the invariance of $\mu_-$ together imply
\[
\begin{split} \sum_{j=1}^N
p_j \rho\bigl(f_j(x),f_j(y)\bigr) &= \sum
_{j=1}^N p_j \min\bigl\{
\mu_{-}\bigl(\bigl[f_j(x),f_j(y)\bigr]\bigr),
\mu _{-}\bigl(\bigl[f_j(y),f_j(x)\bigr]
\bigr)\bigr\}
\\
&= \sum_{j=1}^N p_j \min
\bigl\{ \mu_{-}\bigl(f_j\bigl([x,y]\bigr)\bigr),
\mu_{-}\bigl(f_j\bigl([y,x]\bigr)\bigr) \bigr\}
\\
&\le \min \Biggl\{\sum_{j=1}^N
p_j\mu_{-}\bigl(f_j\bigl([x,y]\bigr)\bigr),
\sum_{j=1}^N p_j
\mu_{-}\bigl(f_j\bigl([y,x]\bigr)\bigr) \Biggr\}
\\
&= \min \bigl\{\mu_{-}\bigl([x,y]\bigr), \mu_{-}\bigl([y,x
\bigr)]) \bigr\} =\rho(x,y) , \end{split} %
\]
which proves that $(F,\bp)$ is non-expansive on average with respect to
$\rho$.
\end{proof}

The following result can be regarded as the heart of the paper.%
\footnote{A similar statement (without proof and stated for systems
where all homeomorphisms preserve orientation) can be found for example
in~\cite{GGKV14}.}

\begin{theorem}\label{teo:lval2}
Let $(F,\bp)$ be an IFS with probabilities of homeomorphisms on $\mathbb
{S}^1$ which is forward minimal and non-expansive on average with
respect to some metric $\rho$.
Then $\rho(Z_n^x,Z_n^y)$ converges almost surely to an $L$-valued
random variable for any $x,y \in\mathbb{S}^1$, where $L=L(F,\rho)$.
\end{theorem}

As an immediate corollary of Proposition~\ref{avcontr} and
Theorem~\ref{teo:lval2} we get the following result. This type of
result is usually referred to as Antonov's theorem (see~\cite{A98},
where
all maps in the IFS are assumed to preserve orientation, see also~\cite
{GGKV14,KN04}).
Also in our generality, the present corollary is not new and follows
(although not explicitly stated)
from results
by Malicet~\cite{M17} who studied
an even more general setting (
without
assuming minimality).

\begin{corollary}
Let $(F,\bp)$ be an IFS with probabilities of homeomorphisms on
$\mathbb{S}^1$ which is forward
and backward
minimal. Then exactly one of the
following cases occurs:
\begin{itemize}
\item[1)] (synchronization) For any $x,y\in\bS^1$ and almost every
$\omega\in\varSigma$ we have\break $d(Z_n(x,\omega),Z_n(y,\omega))\to0$ as
$n\to\infty$.
\item[2)] (factorization) There exists a positive integer $k \geq2$
and a
homeomorphism $\varPsi\colon\bS^1\to\bS^1$ of order $k$ (that is,
$\varPsi^k={\text id}$) which commutes with all $f_j$. Moreover, there is
a naturally associated IFS $\check{F}=\{\check{f}_j\}$
where each map $\check{f}_j$ is a
topological factor%
\footnote{We call a map $g\colon\bS^1\to\bS^1$ a \emph{topological
factor} of $f\colon\bS^1\to\bS^1$ if there exists a continuous
surjective map $\pi\colon\bS^1\to\bS^1$ such that $\pi\circ f=g\circ
\pi$.}
(with a common factoring map) of the corresponding map $f_j$ of $F$
such that $(\check{F},\bp)$ has the synchronization property claimed in item
1).
\item[3)] (invariance) All maps $f_j$ are conjugate (with a common
conjugation map) to an isometry (with respect to $d$). There exists a
probability measure which is\vadjust{\eject} invariant for all maps $f_j$,
$j=1,\ldots,N$, and hence also uniquely invariant for $(F,\bp)$.
\end{itemize}
\end{corollary}

\begin{proof}
Apply Proposition~\ref{avcontr} to $(F,\bp)$ and consider the
homeomorphism $\varPhi_-\colon\bS^1\to\bS^1$, and the metric $\rho$ such that
$(F,\bp)$ is non-expansive on average with respect to $\rho$.
Consider the IFS $(G,\bp)$, conjugate to $(F,\bp)$ through the
conjugating map $\varPhi_-$, which is non-expansive on average with
respect to $d$ and recall $L=L(F,\rho)=L(G,d)$.
We consider three cases:

\noindent{\textbf{Case $L=\{0\}$.}} By Theorem \ref{teo:lval2}, we
have $\rho(Z_n^x,Z_n^y) \rightarrow0$ a.s. and thus $d(Z_n^x,Z_n^y)
\rightarrow0
$ a.s., proving item 1).

\noindent{\textbf{Case $L$ finite and nontrivial.}} By Lemma~\ref{lem:L},
$L(G,d)=\{0,1/k,\ldots,{\lfloor{k/2} \rfloor}/k\}$ for some $k\ge2$.
By Remark \ref{rem:00} applied to $(G,d)$, with
$\check{f}_j(x)=\tilde
g_j(x):=k(g_j(x/k)\mod1/k)$ we have $\check{f}_j\circ\varPsi=\varPsi\circ
f_j$, where $\varPsi=\pi\circ\varPhi_-^{-1}$ with $\pi(x)=kx\mod1$, and the
IFS $(\check{F},\bp)$ satisfies $L(\check{F},d)= L(\widetilde G,d)=\{0\}$.

Since by Lemma \ref{lem:1kinv-b} we have
$\varPhi_-^{-1}(R_{1/k}(x))= (\varPhi_-^{-1}(x) +1/k) \mod1$, it follows that
\[
\begin{split} \varPsi\bigl(R_{1/k}(x)\bigr) &= \bigl(\pi
\circ\varPhi_-^{-1}\bigr) \bigl( R_{1/k}(x)\bigr)
\\
&= \pi\bigl(\bigl(\varPhi_-^{-1}(x)+1/k\bigr) \mod1 \bigr)=\bigl(\pi
\circ\varPhi_-^{-1}\bigr) (x)= \varPsi(x), \end{split} %
\]
and thus
$\varPsi$ 
is an order $k$ homeomorphism having the claimed properties, proving
item 2).

\noindent{\textbf{Case $L$ infinite.}} By Lemma~\ref{lem:L}, we have
$L(G,d)=[0,1/2]$.
All maps in $G$ are thus isometries (with respect to $d$) and hence
simultaneously preserve the Lebesgue measure. The measure
$\mu_+:=(\varPhi_-^{-1})_\ast\mu_{Leb}$ is invariant for all maps of $F$,
and by Lemma \ref{lem:1kinv-b} uniquely invariant for
$(F,\bp)$, proving item 3).
\end{proof}

\begin{remark}
IFSs with nontrivial $L$ can be regarded as degenerated systems. For a
typical system satisfying the conditions of Theorem~\ref{teo:lval2} we
thus have that $\rho(Z_n^x,Z_n^y) \rightarrow0$ as $n \rightarrow
\infty$ a.s. for any $x,y \in\mathbb{S}^1$. Using techniques from
~\cite[Theorem D]{M17} it
seems plausible that it should be possible to prove that
convergence is exponential (see also~\cite{LeJan87}), and that
$(F,\bp)$ is contractive on average with respect to some metric in this case.
\end{remark}

\begin{proof}[Proof of Theorem \ref{teo:lval2}]
Let $(F,\bp)$ be a forward minimal IFS which is non-expansive on
average with respect to $\rho$. Let $\mathscr{F}_n$ be the sigma field
generated by $I_1,\ldots,I_n$.
Fix $x,y\in\mathbb{S}^1$.
Note that $Z_n^x$ and $Z_n^y$ are both measurable with respect to
$\mathscr{F}_n$ and
\begin{eqnarray*}
\bE\bigl( \rho\bigl(Z_{n+1}^x,Z_{n+1}^y
\bigr) | \mathscr{F}_n\bigr) &=& \bE\bigl( \rho\bigl( {f}_{I_{n+1}}
\bigl(Z_n^x\bigr), {f}_{I_{n+1}}\bigl(Z_n^y
\bigr)\bigr) | \mathscr {F}_n \bigr)
\\
&=& \sum_{j=1}^N p_j \rho
\bigl( {f}_{j}\bigl(Z_n^x\bigr),
{f}_{j}\bigl(Z_n^y\bigr)\bigr) \le \rho
\bigl(Z_n^x,Z_n^y\bigr),
\end{eqnarray*}
so the stochastic sequence $(\rho( Z_n^x,Z_n^y))_{n \geq0}$ is a
bounded super-martingale with respect to the filtration $\{\mathscr
{F}_n\}$.
By the Martingale convergence theorem it follows that
$\rho(Z_n^x,Z_n^y) \stackrel{\text{a.s.}}{\rightarrow} \xi$ as $n
\rightarrow\infty$ for some random variable $\xi=\xi^{x,y}$.\vadjust{\eject}

Let $L=L(F,\rho)$.
We will now show that $\xi$ is $L$-valued a.s., that is, we will show
that the distance between any two points $a,b \in\mathbb{S}^1$
with $\rho(a,b)=\xi(\omega)$ is preserved by all the maps in $F$ for
$P$ a.a. $\omega\in\varSigma$.

We will show that any two points $a,b \in\bS^1$ with $\rho(a,b)=\xi
(\omega)$ can simultaneously be (almost) reached by $\{Z_n(x,\omega
),Z_n(y,\omega)\}$ followed by an application of an arbitrary map for
infinitely many $n$ and that this leads to a contradiction if the
distance between some points with distance $\xi(\omega)$ is not
preserved by all maps in $F$ for a typical $\omega$.

Let us first prove the following claim that for any $z$ and any index
$j$, any open set in $\bS^1$ will be visited followed by an application
of the map $f_j$ infinitely many times by trajectories $(Z_n(z,\omega
))_{n \geq0}$ corresponding to typical realizations $\omega$.

\begin{claim}
For any $z\in\bS^1$ and any open set $O\subset\bS^1$ and any $j\in\{
1,\ldots,N\}$ we have
\[
P(\varOmega)=1,\quad \text{ where }\quad \varOmega := \bigcap
_{m=1}^\infty\bigcup_{n=m}^\infty
\bigl\{\omega: Z_n(z,\omega) \in O, \omega_{n+1}=j\bigr\}.
\]
\end{claim}
\begin{proof}
Let $z\in\bS^1$. Consider an open set $O\subset\bS^1$ and an index
$j\in\{1,\ldots,N\}$. By forward minimality, for every $q \in\mathbb
{S}^1$ there exists some positive integer $n_q$ and some $c_q>0$, such that
\begin{equation}
\label{eq:continuityyy} P\bigl(Z_{n_q}^q \in O, I_{n_q+1}=j
\bigr)=P\bigl(Z_{n_q}^q \in O\bigr)P( I_{n_q+1}=j)
>c_q>0.
\end{equation}
Considering the left hand side expression in~\eqref{eq:continuityyy} as
a function of $q$, by continuity (recall the weak Feller property)
one concludes that there exists an open set $O_q$ containing $q$ and
some positive integer $n_q$ and some $c_q'>0$
\[
P\bigl(Z_{n_q}^z \in O, I_{n_q+1}=j\bigr)
>c_q'>0
\]
for any $z \in O_q$. Thus, by compactness, there exists a positive
integer $N$ such that
\[
\inf_{q \in\mathbb{S}^1} P\bigl(Z_n^q \in O,
I_{n+1}=j \text{ for some } n < N\bigr)=:s>0.
\]

Let
\[
\begin{split} A_m &:=\bigl\{ \omega\colon
Z_{n}(z,\omega) \in O,\omega_{n+1}=j, \text{ for some } n\in
\bigl\{mN,\ldots,(m+1)N-1\bigr\}\bigr\}
\\
&= \bigl\{ \omega\colon Z_{n-mN}\bigl(Z_{mN}(z,\omega),
\sigma^{mN}(\omega)\bigr) \in O, \omega_{n+1}=j,
\\
&\phantom{=\{\omega:Z_{n}(z,\omega) \in O,\omega_{n+1}=j, } \text{ for some } n\in\bigl\{mN,\ldots, (m+1)N-1 \bigr\}\bigr\}
. \end{split} %
\]
For any $m\ge1$ we have
\[
\begin{split} P(A_m) &\geq\inf_{q \in\mathbb{S}^1} P
\bigl( \bigl\{ \omega\colon Z_{n-mN}\bigl(q,\sigma^{mN}(\omega)
\bigr) \in O,\omega_{n+1}=j,
\\
&\phantom{=\{\omega:Z_{n}(z,\omega) \in O,\omega_{n+1}=j, } \text{ for some } n\in\bigl\{mN,\ldots, (m+1)N-1 \bigr\}\bigr\}
\bigr)
\\
&= \inf_{q \in\mathbb{S}^1} P\bigl(Z_n^q \in
O,I_{n+1}=j \text{ for some } n < N\bigr)=s>0 \end{split} %
\]
and hence $P(A_m^c) \le1-s$. More generally, we can similarly show that
\[
P \Biggl(\bigcap_{m=j}^k
A_m^c \Biggr) \leq(1-s)^{k-j},
\]
for any $j<k$, which implies
\[
P \Biggl(\bigcap_{j=1}^\infty\bigcup
_{m=j}^\infty A_m\Biggr) =1-P \Biggl(
\bigcup_{j=1}^\infty\bigcap
_{m=j}^\infty A_m^c \Biggr)
=1.
\]
This implies the assertion.
\end{proof}
We can now choose $\varOmega$ with $P(\varOmega)=1$ such that for any $\omega
\in\varOmega$, for any \emph{a priori} fixed index $j$, the trajectory
$(Z_n(x,\omega))_{n \geq0}$ visits infinitely many times any open
interval followed by an application of $f_j$.
Indeed, let $\{a_k\}_k$ be a dense set in $\bS^1$ and for every index
pair $(k,\ell)\in\bN^2$ let $\varOmega^j_{k,\ell}$ be the set provided by
the Claim for the point $x$, an index $j$, and the open set $O_{k,\ell
}=(a_k-1/\ell,a_k+1/\ell)$. Let
\[
\varOmega := \bigcap_{j=1}^N\bigcap
_{k\in\bN}\bigcap_{\ell\in\bN}
\varOmega^j_{k,\ell}
\]
and note that $P(\varOmega)=1$.

By the above, without loss of generality, we can also assume that
$\varOmega$ is such that for every $\omega\in\varOmega$ we have $\rho
(Z_n(x,\omega),Z_n(y,\omega)) \rightarrow\xi(\omega)$ as $n
\rightarrow\infty$.

Fix $\omega\in\varOmega$.
Let $a,b,c$ be points in $\bS^1$, with $\rho(a,b)=\rho(a,c)=\xi(\omega
)$, where
$b$ is obtained from $a$ by a clockwise rotation and $ c$ is obtained
from $a$ by a counter-clockwise rotation.
Note that if $0<\xi(\omega)<1/2$ then the points $a,b,c$ will be
distinct, and
otherwise $b=c$.
By definition of $\varOmega$ we know that if $O_a$ is an open set
containing $a$,
$O_b$ is an open set containing $b$, and $O_c$ is an open set
containing $c$ then there are infinitely many $n$ such that
$Z_n(x,\omega) \in O_a$ and either $Z_n(y,\omega) \in O_b$ or
$Z_n(y,\omega) \in O_c$. We say that $a$
is \emph{clockwise nice} if for arbitrarily small open sets $O_a$ and
$O_b$ containing $a$ and $b$, respectively either $Z_n(x,\omega) \in
O_a$ and $Z_n(y,\omega) \in O_b$ simultaneously or $Z_n(y,\omega) \in
O_a$ and $Z_n(x,\omega) \in O_b$ simultaneously for infinitely many
$n$, and \emph{counterclockwise nice} if for arbitrarily small open
sets $O_a$ and $O_b$ containing $a$ and $b$, respectively either
$Z_n(x,\omega) \in O_a$ and $Z_n(y,\omega) \in O_c$ simultaneously or
$Z_n(y,\omega) \in O_a$ and $Z_n(x,\omega) \in O_c$ simultaneously for
infinitely many $n$. We call $a$ \emph{nice} if $a$ is both clockwise
nice and counterclockwise nice.

\begin{claim}
Any $a \in\bS^1$ is nice.
\end{claim}

\begin{proof}
We first prove that there exist both clockwise nice and
counterclockwise nice points.
Indeed, by definition of $\varOmega$, any $a \in\bS^1$ is either
clockwise nice, counterclockwise nice, or nice. By contradiction,
suppose that all points $a\in\bS^1$ are only clockwise nice (the case
that all points are only counterclockwise nice is analogous). Then, in
particular, a given point $a$ and the point $c$ obtained from a
counterclockwise rotation of $a$ would both be only clockwise nice. But
$c$ being clockwise nice would imply that $a$ is counterclockwise nice,
contradiction.

Thus, there exist points of either type which are arbitrarily close to
each other. Hence,
there exists at least one point in $\bS^1$ which is nice.

By definition of $\varOmega$ it follows that nice points are mapped to
nice points by all maps,
so by forward minimality it follows
that every point in $\bS^1$ is nice.
\end{proof}

Let us now prove that the distance between any two points $a,b\in\bS^1$
with\break $\rho(a,b)=\xi(\omega)$ is preserved by all the maps in $F$.
Arguing by contradiction, suppose that $\xi(\omega) \notin L$, and
consider an interval $[a,b]$ with $\rho(a,b)=\xi(\omega)$
such that for some $j\in\{1,\ldots,N\}$ we have
\[
\rho(a,b)\neq \rho \bigl(f_j(a), f_j(b)\bigr).
\]
By continuity of $f_j$, there exist open intervals $O_a$ and $O_b$
containing $a$ and $b$, respectively and some positive number
$\varepsilon$ such that for any $a' \in O_a$ and any $b' \in O_b$ we have
\[
\big|\rho\bigl(a',b'\bigr) - \rho \bigl(f_j
\bigl(a'\bigr),f_j\bigl(b'\bigr)\bigr)\big| >
\varepsilon.
\]

By choice of $\varOmega$ and the fact that $a$ is nice, there exist
arbitrary large integers $n$ such that either $Z_n(x,\omega) \in O_a$,
and $Z_n(y,\omega) \in O_b$ simultaneously or $Z_n(y,\omega) \in O_a$,
and $Z_n(x,\omega) \in O_b$ simultaneously and
$I_{n+1}(\omega)=\omega_{n+1}=j$. Hence
\[
\big|\rho\bigl(Z_n(x,\omega),Z_n(y,\omega)\bigr) -\rho
\bigl(Z_{n+1}(x,\omega),Z_{n+1}(y,\omega)\bigr)\big| > \varepsilon,
\]
contradicting the assumption that $\omega\in\varOmega$.

This completes the proof that for any $x,y \in\mathbb{S}^1$, $\rho
(Z_n^x,Z_n^y)$ converges almost surely to an $L$-valued random variable.
\end{proof}

The following result about uniqueness of invariant probability measures
is not new and was, to the best of our knowledge, first proved in \cite
{M17}. A simple direct proof based on equicontinuity was recently
presented in \cite{SZ17}. Note that equicontinuity of $\{T^n h \}$,
where $T^n h(x)=\int h(Z_n(x,\omega))\, d P(\omega)$ for any Lipschitz
continuous function $h\colon\mathbb{S}^1\to\bR$,
follows trivially from Proposition \ref{avcontr}. Indeed, if $\rho$ is
the metric of Proposition \ref{avcontr}, then $\int\rho(Z_n(x,\omega
),Z_n(y,\omega)) \,dP(\omega) \leq\rho(x,y)$.
For completeness we will show that uniqueness of invariant probability
measures is also a very simple consequence of Theorem \ref{teo:lval2}.

\begin{corollary}\label{cor:lem:1kinv}
Any IFS $(F,\bp)$ with probabilities of homeomorphisms on $\mathbb
{S}^1$ which is forward and backward minimal has a unique invariant
probability measure $\mu_+$.
\end{corollary}

\begin{proof}
Let $\mu_-$ be an invariant probability measure for $(F^{-1},\bp)$ and
define the metric $\rho$ by $\rho(x,y):=\min\{ \mu_{-}([x,y]), \mu
_{-}([y,x])\}$.
By Proposition~\ref{avcontr}, the IFS $G=\{g_j\}_j$ defined by
$g_j:=\varPhi_-\circ f_j\circ\varPhi_-^{-1}$, where $\varPhi_{-}(x)=\mu
_-([0,x])$, with probabilities $\bp$ is non-expansive on average with
respect to $d$ and we have $L:=L(G,d)=L(F,\rho)$.

By Theorem \ref{teo:lval2}, with $Z_n$ as in~\eqref{def:Zn} and
$W_n:=\varPhi_{-} \circ Z_n \circ\varPhi_{-}^{-1}$, we have that
$d(W_n^x,W_n^y)$ converges almost surely to an $L$-valued random
variable as $n\to\infty$, for any $x,y \in\mathbb{S}^1$.

We are now going to show that there is a unique invariant probability
measure $\nu_+$ for the IFS $(G,\bp)$. This will imply that $\mu_+:=
(\varPhi_{-}^{-1})_\ast\nu_+$
is the unique invariant probability measure for $(F,\bp)$.\vadjust{\eject}

Let us divide the proof into cases:

\smallskip\noindent\textbf{Case $L=\{0\}$.}
Consider first the (generic) case $L=\{0\}$. Thus, $d(W_n^x,W_n^y)
\rightarrow0$ as $n\to\infty$ a.s.\ for any $x,y\in\bS^1$.
Let $\nu_+$ be an invariant probability measure for $(F,\bp)$, that
is, a stationary distribution for $(W_n^x)_{n \geq0}$ (recall Remark
\ref{rem:existamea}).
For any $x,y \in\mathbb{S}^1$ and for any continuous $h \colon\mathbb
{S}^1\to\bR$,
by Lebesgue's dominated convergence theorem
\[
T^n h(x)- T^n h(y) =\int_\varSigma h
\bigl(W_n(x,\omega)\bigr)dP(\omega)-\int_\varSigma h
\bigl(W_n(y,\omega )\bigr)dP(\omega) \rightarrow0
\]
as $n \rightarrow\infty$, and thus
by invariance of $\nu_+$ we have
\[
\biggl\lvert T^n h(x)- \int h \,d \nu_+ \biggr\rvert = \biggl\lvert
T^n h(x)- \int T^n h \,d \nu_+ \biggr\rvert \leq\int\bigl
\lvert T^n h(x)- T^n h(y) \bigr\rvert\,d \nu_+(y)
\]
and by Lebesgue's dominated convergence theorem the latter tends to $0$
as $n \rightarrow\infty$. This implies that $\nu_+$ must be attractive
and thus unique (recall Remark \ref{rem:unistamea}).

\smallskip\noindent\textbf{Case $L=\{0,1/k,\ldots,\lfloor{k/2}\rfloor
/k\}$ for some $k\ge2$.}
By Lemma~\ref{lem:1kinv-b} all invariant probability measures for
$(G,\bp)$
are $1/k$-invariant. By contradiction, suppose that there are
two distinct invariant probability measures $\nu_+^1$ and $\nu_+^2$ for
$(G,\bp)$. Hence, if $X$ and $Y$ are two random variables with
distribution $\nu_+^1$ and $\nu_+^2$ respectively, independent of $\{
I_n\}$, then $W_n^X \mod1/k$, and $W_n^Y$ mod $1/k$ will also have
distinct distributions for any fixed $n \geq0$, by $1/k$--invariance
of $\nu_+^1$ and $\nu_+^2$.
The latter is however impossible since the IFS $\widetilde G=\{\tilde
g_j\}$ defined by $\tilde g_j(x)=k (g_j(x/k)\mod 1/k)$, $j=1, \ldots,
N$, satisfies $L(\widetilde G,d)=\{0\}$ (recall Remark \ref{rem:00})
and therefore the distribution of $W_n^X\mod 1/k$ converges to the
same limit as the limiting distribution of $W_n^Y\mod1/k$, as $n
\rightarrow\infty$. The invariant probability measure, $\nu_+$, is
therefore unique.

\smallskip\noindent\textbf{Case $L=[0,1/2]$.}
In this case, by Lemma~\ref{lem:L} all maps in $G$ are isometries (with
respect to $d$). By Lemma~\ref{lem:1kinv-b}, any invariant probability
measure is $s$-invariant for any $s\in[0,1/2]$, which implies that $\nu
_+$ must be the Lebesgue measure.
\end{proof}

By applying Breiman's ergodic theorem
for Feller chains with a unique stationary distribution starting at a
point (see, for example,~\cite{B60} or \cite{MT93}), we get the
following result.
Let $\delta_{x}$ denote the Dirac measure concentrated in the point
$x\in\bS^1$,
and let
\[
\mu_n^x(\omega) = \frac{1}{n} \sum
_{k=0}^{n-1} \delta_{Z_k(x,\omega)},
\]
denote the empirical distribution along the trajectory starting at $x
\in\mathbb{S}^1$ determined by $\omega\in\varSigma$ at time $n-1$.

\begin{corollary} \label{0218}
Let $(F,\bp)$ be an IFS with probabilities of homeomorphisms on $\mathbb
{S}^1$ which is forward and backward minimal and let $\mu_{+}$ denote
its unique invariant probability measure.
Then
$\mu_n^x(\omega)$ converges to $\mu_+$ (in the weak$\ast$ sense) $P$
a.s. for any $x \in \mathbb{S}^1$.
\end{corollary}

\begin{remark}
Corollary \ref{0218} slightly generalizes \cite[Proposition 16]{SZ16}
where a direct proof is given and the additional hypotheses that all
maps in the IFS preserve orientation and that one map is minimal are assumed.
\end{remark}

Let $\stackrel{\text{d}}{\rightarrow}$ denote convergence in distribution.
We are now ready to state our first result about invariant
measures/stationary distributions for the IFS with probabilities
generated by the inverse maps.

\begin{proposition}\label{lval3}
Let $(F,\bp)$ be an IFS with probabilities of homeomorphisms on
$\mathbb{S}^1$ which is forward minimal and non-expansive on average
with respect to $d$.
Assume that some map $f_j$ is not an isometry (with respect to $d$). Then
$L(F,d)=\{0,1/k,\ldots,\break\lfloor{k/2} \rfloor/k \}$ for some $k \geq1$
and for any $1/k$-invariant
nonatomic and fully supported random variable $X$ on $ \mathbb{S}^1$,
independent of $(I_n)_{n \geq0}$ we have
\[
\widehat{Z}^{-}_n(X,\omega) \stackrel{\text{d}} {
\rightarrow} {\bf \widehat Z}^-(\omega)
\]
as $n \rightarrow\infty$ for $P$ a.a. $\omega\in\varSigma$, where ${\bf
\widehat Z}^-(\omega)$ is a random variable with distribution
\[
\mu_\omega^- =\frac{1}{k} \sum_{i=0}^{k-1}
\delta_{\frac{1}{k}(i+\widehat Z^-(\omega))}
\]
for some random variable $\widehat{Z}^{-}\colon\varSigma\rightarrow
\mathbb{S}^1$ and $\mu_{ \omega}^-=(f_{\omega_1}^{-1})_\ast\mu_{\sigma
(\omega)}^-$ for $P$ a.a. $\omega\in\varSigma$.

Thus, the measure $\mu_-$ given by
\[
\mu_- := \int\mu_\omega^-\, dP(\omega)
\]
is the unique invariant probability measure for $(F^{-1},\bp)$.
\end{proposition}

\begin{proof}
Let $L=L(F,d)$. By Lemma~\ref{lem:L} together with our hypotheses, we
have $L=\{0,1/k,\ldots,\lfloor{k/2}\rfloor/k\}$ for some $k\ge1$.
Hence, if $d(x,y)=s \in L$, then
\[
d\bigl( f_j(x),f_j(y)\bigr)=s
\]
for all $j=1,\ldots,N$ and thus
\[
d\bigl( Z_n(x,\omega),Z_n(y,\omega)\bigr)=s
\]
for any $\omega\in\varSigma$ and $n \geq0$.

Let us denote by $Z_n$ and $\widehat Z_n^-$ the sequences defined
in~\eqref{def:Zn} and~\eqref{def:hatZn-}, respectively. By Theorem~\ref
{teo:lval2} we have that $d(Z_n^x,Z_n^y)$ converges almost surely to an
$L$-valued random variable as $n\to\infty$.

Given $\omega\in\varSigma$, let
\[
\widehat{Z}^{-}(\omega) := k \sup\bigl\{y\colon \big|Z_n
\bigl([0,y],\omega\bigr)\big| \rightarrow0, \text{ as } n \rightarrow\infty\bigr\},
\]
where $\lvert\cdot\rvert$ denotes the length of an interval and where
we use the notation
\[
Z_n\bigl([0,y],\omega\bigr) := \bigl\{Z_n(z,\omega)
\colon z\in[0,y]\bigr\}.
\]
Note that
$y\mapsto|Z_n([0,y],\omega)|$ is an increasing function, for each
fixed $n$ and $\omega$. Further, $|Z_n([0,y],\omega)|$ converges to
an element in $\{0,1/k,\ldots,1\}$
as $n \rightarrow\infty$, for any $y \in\mathbb{S}^1$ for $P$ a.a.
$\omega\in\varSigma$. Indeed, this follows from the fact that
$d(Z_n^x,Z_n^y)$\vadjust{\eject} converges to an element of $L$ and the fact that
$x\mapsto Z_n^x$ is a random homeomorphism. So $\widehat{Z}^{-}\colon
\varSigma\to\bS^1$ is a well-defined random variable.

Let $m$ be an arbitrary $1/k$-invariant nonatomic probability measure
fully supported on $\mathbb{S}^1$.
Note that if $I$ is an interval of length $i/k$, then $m(I)=i/k$ for
any $0 \leq i \leq k$.
If $x \notin\{\widehat{Z}^{-}(\omega)/k,
(\widehat{Z}^{-}(\omega)+1)/k,\ldots,(\widehat{Z}^{-}(\omega)+(k-1))/k\}
$, then
\[
\begin{split} m& \bigl( \bigl\{ y \in\mathbb{S}^1\colon
\widehat Z ^-_n(y,\omega) \leq x \bigr\} \bigr) = m \bigl( \bigl\{ y
\in\mathbb{S}^1\colon \widehat Z ^-_n(y,\omega) \in[0,x]
\bigr\} \bigr)
\\
&= m \bigl( \bigl\{ y \in\mathbb{S}^1\colon Z_n\bigl(
\widehat Z ^-_n(y,\omega),\omega\bigr) \in Z_n\bigl([0,x],
\omega\bigr) \bigr\} \bigr)
\\
&= m \bigl( \bigl\{ y \in\mathbb{S}^1\colon y \in{Z_n}
\bigl([0,x],\omega\bigr) \bigr\} \bigr)
\\
&\rightarrow %
\begin{cases}
0 & \text{ if }x < \displaystyle\frac{\widehat{Z}^{-}(\omega)}{k}, \\
\displaystyle\frac{i}{k} & \text{ if }\displaystyle\frac{\widehat
{Z}^{-}(\omega)+(i-1)}{k} < x <
\displaystyle\frac{\widehat{Z}^{-}(\omega)+i}{k}, \,\,1 \leq i \leq
k-1, \\
1 & \text{ if }x> \displaystyle\frac{\widehat{Z}^{-}(\omega)+(k-1)}{k}
\end{cases} %
\end{split} %
\]
as $n \rightarrow\infty$ for $P$ a.a. $\omega\in\varSigma$.
Thus, if $X$ is an $m$-distributed random variable on $\mathbb{S}^1$,
independent of $(I_n)_{n \geq0}$ and ${\bf\widehat Z}^-(\omega)$
has distribution
\[
\mu_\omega^- =\frac{1}{k} \sum_{i=0}^{k-1}
\delta_{{(i+\widehat Z^-(\omega))}/{k}} \quad\text{ for } P \text{ a.a. } \omega\in\varSigma
\]
then $m(\widehat{Z}^{-}_n(X,\omega) \leq x) \rightarrow m( {\bf
\widehat Z}^-(\omega) \leq x)$ as $n \rightarrow\infty$ if $x$ is a
continuity point of the cumulative distribution function of ${\bf
\widehat Z}^-(\omega)$ (for $P$ a.a. $\omega\in\varSigma$). Thus,
$\widehat{Z}^{-}_n(X,\omega)$ converges in distribution to ${\bf
\widehat Z}^-(\omega)$ as $n \rightarrow\infty$ for $P$ a.a. $\omega
\in\varSigma$. By taking limits in the equality $\widehat
{Z}_n^{-}(X,\omega)= f_{\omega_1}^{-1}(\widehat{Z}_{n-1}^{-}(X,\sigma
(\omega)))$,
it therefore follows that
\[
{\bf\widehat Z}^-( \omega) =f_{\omega_1}^{-1} \bigl({\bf\widehat
Z}^-(\sigma\omega)\bigr),
\]
for $P$ a.a. $\omega$.
Thus if $\mu_\omega^{-}$ denotes the distribution of ${\bf\widehat
Z}^-( \omega)$, then
\[
\mu_\omega^{-} =\bigl(f_{\omega_1}^{-1}
\bigr)_\ast\mu_{\sigma(\omega)}^-.
\]
for $P$ a.a. $\omega$.

By integrating both sides of this equality with respect to $P$ (recall
that $P$ is a Bernoulli measure determined by a probability vector $\bp
=(p_1,\ldots,p_N)$) we thus obtain that
\begin{eqnarray*}
\mu_- &:=& \int\mu_\omega^-\, dP(\omega)= \sum
_{j=1}^N \int_{\omega: \omega
_1=j}
\bigl(f_{\omega_1}^{-1}\bigr)_\ast\mu_{\sigma(\omega)}^-\,
dP(\omega)
\\
&=& \sum_{j=1}^N \int
_{\omega: \omega_1=j} \bigl(f_{j}^{-1}
\bigr)_\ast\mu_{\sigma
(\omega)}^-\, dP(\omega)
\\
&=& \sum_{j=1}^N \int
_{\omega: \omega_1=j} \bigl(f_{j}^{-1}
\bigr)_\ast \mu_{-} \, dP(\omega)
\\
&=& \sum_{j=1}^N p_j
\bigl(f_{j}^{-1}\bigr)_\ast\mu_{-}\,
\end{eqnarray*}
and
$\mu_-$ is therefore invariant for $(F^{-1},\bp)$.
Since by construction $\mu_\omega$ is independent of $X$, it follows
that $\mu_-$ is indeed uniquely invariant.
\end{proof}

Using Propositions~\ref{avcontr} and~\ref{lval3} we get the following corollary.
\begin{corollary} \label{Corr4}
Let $(F,\bp)$ be an IFS with probabilities of homeomorphisms on
$\mathbb{S}^1$ which is forward and backward minimal. Assume
that not all maps in $F$ are conjugate (with a common
conjugation map) to an isometry (with respect to $d$).
Let $\mu_-$ be an invariant probability measure for $(F^{-1},\bp)$, and
let $k$ be the largest integer such that
$\mu_-$ is $1/k$-invariant.
Conclusion: if $X$ is a $\mu_{-}$-distributed random variable,
independent of $(I_n)_{n \geq0}$, then
\begin{equation}
\label{FurConv} \widehat{Z}^{-}_n(X,\omega) \stackrel{
\text{d}} {\rightarrow} {\bf\widehat Z}^-(\omega),
\end{equation}
as $n \rightarrow\infty$ for $P$ a.a. $\omega\in\varSigma$,
where ${\bf\widehat Z}^-(\omega)$ is a random variable with distribution
$\mu_\omega$, uniformly distributed on $k$ distinct points, and
satisfying $\mu_{ \omega}^-=(f_{\omega_1}^{-1})_\ast\mu_{\sigma(\omega
)}^-$ for $P$ a.a. $\omega\in\varSigma$.
It therefore follows that $\mu_-$ is unique and given by $\mu_-= \int
\mu_\omega^-\, dP(\omega)$.
\end{corollary}

\begin{remark}
Convergence in (\ref{FurConv}) also follows from Furstenbergs
martingale argument \cite{Fur:73}, but here we say more about the
limit: The limit is $1/k$-invariant and independent of $X$ (this
implies that $\mu_{-}$ is uniquely invariant) and the limiting fiber
measures $\mu_\omega$ are uniform and supported on sets of size $k$.
\end{remark}

\begin{proof}
Let $\mu_-$ be an invariant probability measure for $(F^{-1},\bp)$. By
Proposition~\ref{avcontr}, the IFS $G=\{g_j\}_j$ defined by $g_j:=\varPhi
_-\circ f_j\circ\varPhi_-^{-1}$, where $\varPhi_{-}(x)=\mu_-([0,x])$, with
probabilities $\bp$ satisfies the hypotheses of Proposition~\ref
{lval3}. Note that $F$ is forward minimal if, and only if, $G$ is.
Let $L(G,d)$ be the corresponding set of simultaneously preserved
distances. By Lemma~\ref{lem:L}, we have $L(G,d)=\{0,1/k,\ldots,\lfloor
k/2\rfloor/k\}$ for some $k\ge1$.

Let us denote by $( \widehat W_n^-)_{n \geq0}$ the sequence for the
IFS $G$
which is analogously defined as in~\eqref{def:hatZn-} for the IFS $F$.
Since $F$ and $G$ are conjugate by means of $\varPhi_{-}$,
it is easy to check that
\[
\widehat Z_n^{-} =\varPhi_{-}^{-1}
\circ\widehat W_n^- \circ\varPhi_{-}.
\]

Note that if $X$ is a $\mu_{-}$-distributed random variable,
independent of $(I_n)_{n \geq0}$, then $Y:=\varPhi_{-}(X)$ is distributed
according to the Lebesgue measure on $\mathbb{S}^1$. Hence, in
particular, it follows that $Y$ is a $1/k$-invariant, nonatomic, and
fully supported random variable on $ \mathbb{S}^1$.

By Proposition~\ref{lval3}, it therefore follows that
\[
\widehat{W}^{-}_n\bigl(\varPhi_{-}(X),\omega
\bigr)= \varPhi_{-}\bigl( \widehat{Z}^{-}_n (X,
\omega)\bigr) \stackrel{\text{d}} {\rightarrow} {\bf\widehat W}^-(\omega)
\]
as $n \rightarrow\infty$ for $P$ a.a. $\omega\in\varSigma$,
where ${\bf\widehat W}^-(\omega)$ is a random variable with distribution
\[
\nu_\omega^- =\frac{1}{k} \sum_{i=0}^{k-1}
\delta_{\frac{1}{k}(i+\widehat W^-(\omega))}
\]
for some random variable $\widehat{W}^{-}\colon\varSigma\rightarrow
\mathbb{S}^1$,
that is, we have
\[
\widehat{Z}^{-}_n(X,\omega) \stackrel{\text{d}} {
\rightarrow} \varPhi^{-1}_{-}\bigl({\bf\widehat W}^-(\omega)
\bigr)
\]
as $n \rightarrow\infty$ for $P$ a.a. $\omega\in\varSigma$.
Thus, if we define
${\bf\widehat Z}^-(\omega):= \varPhi^{-1}_{-}({\bf\widehat W}^-(\omega))$,
then this random variable has distibution
$\mu_\omega(\cdot)=\nu_\omega(\varPhi_{-}(\cdot))$.
Moreover, the measure $\mu_-$ given by
\[
\mu_- := \int\mu_\omega^-\, dP(\omega)
\]
is the unique invariant probability measure for $(F^{-1},\bp)$.
\end{proof}

\begin{remark}
If $(F,\bp)$ is both forward and backward minimal, then
by applying the above Corollary to $(F^{-1},\bp)$ we obtain an
alternative proof of uniqueness for $\mu_{+}$ under these assumptions.
\end{remark}

\section*{Acknowledgement}
We are grateful to the referees whose comments on an earlier draft of
the paper led to a substantially improved revised paper. KG also thanks
Anton Gorodetski and Victor Kleptsyn for inspiring discussions.


%
\end{document}